\documentclass[a4paper]{amsart}
\usepackage{url}
\usepackage[all]{xy}
\usepackage[mathscr]{eucal}
\usepackage{amsmath,amssymb,amsfonts}
\usepackage{mathrsfs,latexsym,amsthm,enumerate}
\newtheorem{theorem}{Theorem}[section]
\newtheorem{lemma}[theorem]{Lemma}
\newtheorem{corollary}[theorem]{Corollary}
\newtheorem{example}[theorem]{Example}

\newtheorem{proposition}[theorem]{Proposition}

\usepackage{stackengine} 
\usepackage{tikz}

\title{Generalizations of free monoids}

\author{Mark V. Lawson}
\address{Mark V. Lawson, Department of Mathematics
and the
Maxwell Institute for Mathematical Sciences, 
Heriot-Watt University,
Riccarton,
Edinburgh EH14 4AS, 
UNITED KINGDOM}
\email{m.v.lawson@hw.ac.uk}

\author{Alina Vdovina}
\address{Alina Vdovina, 
Department of Mathematics,
The City College of New York,
160 Convent Avenue,
New York, NY 10031,
USA}
\email{avdovina@ccny.cuny.edu}

\begin{document} 
\dedicatory{This paper is dedicated to Gracinda Gomes, colleague and friend, on the occasion of her retirement.}

%%%%%%%%%%%%%%%%%%%%%%%%%%%%%%%%%%%%%%%%%%%%%%%%%%%%%%%%%%%%%%%%%%%%%%%%%%%%%%%%%%%%%%%%%%%%%%
\begin{abstract} 
We generalize free monoids by defining $k$-monoids. 
These are nothing other 
than the one-vertex higher-rank graphs used in $C^{\ast}$-algebra theory with the cardinality requirement waived.
The $1$-monoids are precisely the free monoids. 
We then take the next step and generalize $k$-monoids
in such a way that self-similar group actions yield monoids of this type.
\end{abstract}
\maketitle

%%%%%%%%%%%%%%%%%%%%%%%%%%%%%%%%%%%%%%%%%%%%%%%%%%%%%%%%%%%%%%%%%%%%%%%%%%%%%%%%%%%%%%%%%%%%%%%%%
\section{Introduction}

The goal of this paper is to generalize free monoids to higher dimensions.
We make little pretence of novelty (except, perhaps, in Section 5 and Example~\ref{ex:av1} and Example~\ref{ex:av2}) 
since the monoids considered in this paper 
are nothing other than the one-vertex higher-rank graphs with the usual cardinality restriction waived.
Higher-rank graphs were introduced in \cite{KP} (formalizing some ideas to be found in \cite{RS})
and the monoids within this class have been considered by a number of authors, such as \cite{DPY, DY2009b}.
They are well-known within the operator algebra community,
but, we maintain, they should also be interesting to those working within semigroup theory.
Whereas free monoids are concretely monoids of strings, our monoids will have elements that
we can regard as `higher-dimensional strings'; for example, in two dimensions
our elements can be regarded as rectangles.
Our generalization of free monoids is called $k$-monoids;
the $1$-monoids will turn out to be precisely the free monoids.
Classes of $k$-monoids were studied in \cite{LV2020} where they were used to construct groups via inverse semigroups.
This work is summarized in Section 4.
Our point of view is that any result for free monoids should be generalized to $k$-monoids.\\

\noindent
{\bf Acknowledgements }The authors would like to thank Aidan Sims for his comments on an earlier draft of this paper.
Whilst this paper was being prepared, the authors were sad to learn of the passing of Iain Raeburn.
Iain contributed greatly to the theory of $C^{\ast}$-algebras in general, and higher-rank graphs in particular.
The authors would also like to thank the anonymous referee for a careful reading
of a draft of this paper and many constructive suggestions.\\

%%%%%%%%%%%%%%%%%%%%%%%%%%%%%%%%%%%%%%%%%%%%%%%%%%%%%%%%%%%%%%%%%%%%%%%%%%%%%%%%%%%%%%%%%%%%%%%%%%%%%%%%%%%%%%%%%%%%%%%%%%%%%%
\section{Free monoids}

You can read all about free monoids in Lallement's book \cite[Chapter 5]{Lallement} but we shall go over what we need here.
Our goal is to motivate the definition of $k$-monoids which will be given in the next section.
Let $A$ be any set, called in this context an {\em alphabet}, whose elements will be called {\em letters.}
We do not need to assume that $A$ is finite and, although it could be empty, that is not a very interesting case.
By a {\em string} over $A$ we mean a finite sequence of elements of $A$.
We shall dispense with brackets and so strings shall simply be written as words over the alphabet $A$.
The empty string is denoted by $\varepsilon$.
The set of all strings over $A$ is denoted by $A^{\ast}$.
The set $A^{\ast}$ becomes a semigroup with operation $\cdot$ when we combine strings via concatenation:
thus $x \cdot y = xy$.
This really makes $A^{\ast}$ into a semigroup and, in fact, a monoid with identity $\varepsilon$.
As usual, we shall omit explicit reference to the semigroup operation $\cdot$.
Monoids that are isomorphic to the monoids $A^{\ast}$ are called {\em free monoids}.
Free monoids are even simpler than free groups since there are no inverses to deal with.
The elements of $A$ are called the {\em generators} of the free monoid.
The simplest interesting free monoids are those with exactly one generator;
such monoids are isomorphic to the monoid $(\mathbb{N},+)$.
So, arbitrary free monoids can be regarded as non-commutative arithmetic.
The case where $A$ is empty is special: the free monoid on no generators is just the one-element monoid.

Free monoids have some important algebraic properties but
to state these, we need some definitions.
We say that a monoid $S$ is {\em equidivisible} if $xy = uv$ in $S$ 
implies there exists an element $t \in S$ such that
either $u = xt$ and $y = tv$ or $x = ut$ and $v = ty$.
If $x$ is any string in a free monoid, denote its {\em length} by $|x|$;
this simply counts the total number of letters of $A$ occurring in $x$ including multiplicities.
Free monoids are {\em cancellative}: this means that if $xy = xv$ then $y = v$
and if $xy = uy$ then $x = u$.
The following was proved as  \cite[Chapter 5, Proposition 1.5]{Lallement}.

\begin{lemma}\label{lem:darwin} 
Free monoids are equidivisible.
\end{lemma}

Free monoids come equipped with a monoid homomorphism $\delta$ which associates with a string its length.
Thus, there is a monoid homomorphism $\delta \colon A^{\ast} \rightarrow \mathbb{N}$ given by $\delta (x) = |x|$.
It is possible to characterize free monoids by means of the properties of the monoid homomorphism $\delta$.
The following \cite[Chapter 5, Corollary 1.6]{Lallement} was first proved by F. W. Levi in \cite{Levi}.

\begin{theorem}[F. W. Levi, 1944]\label{them:levi} A monoid $S$ is free if and only if $S$ is equidivisible
and there exists a monoid homomorphism $\theta \colon S \rightarrow \mathbb{N}$ such that $\theta^{-1}(0)$
is just the identity of $S$.
\end{theorem}

\begin{quote}
{\em We shall use the above theorem to obtain a new characterization of free monoids that will motivate this paper.}\\
\end{quote}

Let $x \in A^{\ast}$.
Then $|x| \in \mathbb{N}$.
Suppose that $m,n \in \mathbb{N}$ such that $m + n = |x|$.
Then there are {\em unique} elements $u,v \in A^{\ast}$ such that $x = uv$
where $|u| = m$ and $|v| = n$.
To see that this is true, just remember that the elements of the free monoid are strings.

More generally, we say that a monoid $S$ has the {\em unique factorization property (UFP)} if it is equipped with
a monoid homomorphism $\theta \colon S \rightarrow \mathbb{N}$ such that
if $\theta (a) = m + n$ then there are unique elements $b,c \in S$ such that $a = bc$ where $\theta (b) = m$ and $\theta (c) = n$.
We now look at the consequences of the (UFP) for the structure of the monoid.

\begin{lemma}\label{lem:basic-UFP} Let $S$ be a monoid and let $\theta \colon S \rightarrow \mathbb{N}$
be a monoid homomorphism that satisfies the (UFP). Then 
\begin{enumerate}
\item $S$ is equidivisible.
\item $\theta^{-1}(0)$ is just the identity of $S$. 
\end{enumerate}
\end{lemma}
\begin{proof} (1) Suppose that $xy = uv$.
We shall compare $\theta (x)$ with $\theta (u)$.
The set $\mathbb{N}$ is linearly ordered, so that either $\theta (x) < \theta (u)$
or $\theta (x) \geq \theta (u)$.
Suppose first that $\theta (x) < \theta (u)$.
Then, we can write $\theta (u) = \theta (x) + n$, for some natural number $n$.
By the (UFP), we can write $u = xt$ where $t \in S$ is the unique element such that $\theta (t) = n$.
It follows that $xy = xtv$.
Take $\theta$ of both sides, and use basic algebra, to get that
$\theta (y) = \theta (t) + \theta (v)$.
By the (UFP), it follows that $y = tv$.
The case where $\theta (x) \geq \theta (u)$ can be handled similarly.

(2) Because $\theta$ is a monoid homomorphism, we have that $\theta (1) = \mathbf{0}$.
Now, let $e \in S$ be such that $\theta (e) = \mathbf{0}$.
But $\theta (e) = \mathbf{0} + \mathbf{0}$.
By the (UFP), we may write $e = e_{1}e_{2}$ uniquely where $\theta (e_{1}) = \mathbf{0}$ and $\theta (e_{2}) = \mathbf{0}$.
But $e = 1e = e1$, also since $1$ is the identity.
It follows by uniqueness that $e_{1} = 1$ and $e_{2} = 1$.
We deduce that $e = 1$.
\end{proof}

We can now provide a different characterization of free monoids
which follows immediately by Theorem~\ref{them:levi} and Lemma~\ref{lem:basic-UFP}.

\begin{theorem}\label{them:one} 
Let $S$ be a monoid equipped with a monoid homomorphism $S \rightarrow \mathbb{N}$
that satisfies the (UFP).
Then $S$ is free. 
\end{theorem}

The above theorem directly motivates the definition of the next section.

%%%%%%%%%%%%%%%%%%%%%%%%%%%%%%%%%%%%%%%%%%%%%%%%%%%%%%%%%%%%%%%%%%%%%%%%%%%%%%%%%%%%%%%%%%%%%%%%%%%%%%%%%%%%%%%%%%%%%%%%%%%%%%%%%%%%%%%%%%%%%%%%%%%%%%%%%%%%%%%%%%%%%%%%
\section{$k$-monoids} 

Theorem~\ref{them:one} sets the scene for what we shall do in this section:
essentially, we shall replace $\mathbb{N}$ by $\mathbb{N}^{k}$.
We shall need a little notation first.
If $\mathbf{m} \in \mathbb{N}^{k}$ then 
$$\mathbf{m} = (m_{1}, \ldots, m_{i}, \ldots, m_{k})$$ 
and we define $\mathbf{m}_{i} = m_{i}$, called the {\em components} of $\mathbf{m}$.
The partial order in $\mathbb{N}^{k}$ is defined componentwise: $\mathbf{m} \leq \mathbf{n}$ if and only if $\mathbf{m}_{i} \leq \mathbf{n}_{i}$ for every $i$ in the range $1 \leq i \leq k$.
The join operation is $(\mathbf{m} \vee \mathbf{n})_{i} = \mbox{max}(m_{i},n_{i})$ and the meet operation is $(\mathbf{m} \wedge \mathbf{n})_{i} = \mbox{min}(m_{i},n_{i})$.
Put $\mathbf{0} = (0, \ldots, 0)$ and $\mathbf{1} = (1, \ldots, 1)$ both elements of $\mathbb{N}^{k}$.
Define $\mathbf{e}_{i}$, where $1 \leq i \leq k$, to be that element of $\mathbb{N}^{k}$ which is zero everywhere except at $i$ where it takes the value $1$.\\

\noindent
{\bf NB:} The set $\mathbb{N}^{k}$  is equipped with a lattice structure,
where the meet and join operations operations are defined as above.\\

\noindent
{\bf Definition. }A monoid $S$ is said to be a {\em $k$-monoid} if there is a monoid homomorphism
$\delta \colon S \rightarrow \mathbb{N}^{k}$ satisfying the {\em unique factorization property (UFP)}:
if $\delta (x) = \mathbf{m} + \mathbf{n}$ then there exist unique elements $x_{1}$ and $x_{2}$ of $S$ 
such that $x = x_{1}x_{2}$ where $\delta (x_{1}) = \mathbf{m}$ and $\delta (x_{2}) = \mathbf{n}$.
We call $\delta (a)$ the {\em size} of $a$.\\

\noindent
{\bf NB:} Observe that a $k$-monoid is defined with respect to a given
monoid homomorphism $\delta \colon S \rightarrow \mathbb{N}^{k}$.\\

Using the terminology we have just introduced, we proved in the previous section that the $1$-monoids are precisely the free monoids.
Thus $k$-monoids really do generalize free monoids.
However, the direct product of free monoids is not usually free,
but one particularly pleasant feature of $k$-monoids is that they are closed under finite direct products.
To see this, let $S$ be a $k$-monoid and let $T$ be an $l$-monoid.
We shall denote their respective monoid homomorphisms by 
$\delta_{S} \colon S \rightarrow \mathbb{N}^{k}$ 
and 
$\delta_{T} \colon T \rightarrow \mathbb{N}^{l}$.
There is a natural isomorphism $\mu \colon \mathbb{N}^{k} \times \mathbb{N}^{l} \rightarrow \mathbb{N}^{k + l}$
where we take the ordered pair $((m_{1}, \ldots, m_{k}),(n_{1},\ldots,n_{l}))$
to the single $(k + l)$-tuple $(m_{1},\ldots, m_{k},n_{1}, \ldots, n_{l})$.
Define $\delta \colon S \times T \rightarrow \mathbb{N}^{k+l}$ 
by $\delta (s,t) = \mu (\delta_{S}(s),\delta_{T}(t))$.
In this way, it is easy to prove that $S \times T$ is a $(k + l)$-monoid.
We have therefore established the following \cite[Proposition 1.8]{KP}.

\begin{lemma}\label{lem:prod} 
If $S$ is a $k$-monoid and $T$ is an $l$-monoid then $S \times T$ is a $k+l$-monoid.
\end{lemma}

The proof of the following is immediate by the above lemma.

\begin{corollary}\label{cor:kfree-is-kmonoid} Let $A_{1}, \ldots, A_{k}$ be $k$ alphabets.
Then $A_{1}^{\ast} \times \ldots \times A_{k}^{\ast}$ is a $k$-monoid.
\end{corollary}

We say that a monoid $S$ is {\em singly aligned} if $aS \cap bS \neq \varnothing$ implies that $aS \cap bS = cS$ for some $c \in S$.
The proof of the following is easy.

\begin{lemma}\label{lem:needed-later} 
The product of two singly aligned monoids is singly aligned.
\end{lemma}

The following is now immediate from the above lemma, Corollary~\ref{cor:kfree-is-kmonoid} and properties of free monoids.

\begin{corollary}\label{cor:kfree-is-kmonoid} Let $A_{1}, \ldots, A_{k}$ be $k$ alphabets.
Then $A_{1}^{\ast} \times \ldots \times A_{k}^{\ast}$ is a singly aligned $k$-monoid.
\end{corollary}

%%%%%%%%%%%%%%%%%%%%%%%%%%%%%%%%%%%%%%%%%%%%%%%%%%%%%%%%%%%%%%%%%%%%%%%%%%%%%%%%%%%%%%%%%%%%%%%%%%%%%%%%%%%%%%%%%%%%%%%%%%
The following portmanteau lemma (proved first in \cite{KP}) summarizes some of the important algebraic properties of $k$-monoids.
Recall that a monoid is said to be {\em conical} if its group of units is trivial.

\begin{lemma}\label{lem:alg-properties}
Let $S$ be a $k$-monoid with identity $1$ 
and with monoid homomorphism $\delta \colon S \rightarrow \mathbb{N}^{k}$.
\begin{enumerate}
\item  $S$ is cancellative.
\item  $\delta (1) = \mathbf{0}$ and is the only element that is mapped by $\delta$ to $\mathbf{0}$.
\item $S$ is conical.
\end{enumerate}
\end{lemma}
\begin{proof}
(1)  Let $z = xy = xv$.
We have that $\delta (y) = \delta (v)$ and so,
by the (UFP), we have that $y = v$.
Now, suppose that $xy = uy$.
We apply the (UFP) again to deduce that $x = u$.

(2) We have that $11 = 1$ and so
$\delta (11) = \delta (1) + \delta (1) = \delta (1)$.
It follows that $\delta (1) = \mathbf{0}$.
Now, suppose that $\delta (a) = \mathbf{0}$ where $a \in S$.
Then $a = a_{1}a_{2}$ where $\delta (a_{1}) = \mathbf{0} = \delta (a_{2})$.
But $a = a1 = 1a$.
We have that $\delta (a) = \mathbf{0} + \mathbf{0}$.
We deduce that $a = 11 = 1$.

(3) Suppose that $xy = 1$.
Then $\delta (x) = \delta (y) = \mathbf{0}$.
By part (2) above both $x$ and $y$ is equal to the identity.
\end{proof}

%%%%%%%%%%%%%%%%%%%%%%%%%%%%%%%%%%%%%%%%%%%%%%%%%%%%%%%%%%%%%%%%%%%%%%%%%%%%%%%%%%%%%%%%%%%%%%%%%%%%%%%%%%%%%%%%%%%%%%%%
The biggest difference between $k$-monoids, where $k \geq 2$, 
and free monoids lies in the fact that, with respect to the usual order, whereas the set $\mathbb{N}$ is linearly ordered
the set $\mathbb{N}^{k}$ is not.
Here is the appropriate analogue of Lemma~\ref{lem:darwin}.

\begin{lemma}\label{lem:levi} Let $S$ be a $k$-monoid.
Let $xy = uv$ and suppose that $\delta (x) \geq \delta (u)$.
Then there is $t \in S$ such that $x = ut$ and $v = ty$.
In particular, if $\delta (x) = \delta (u)$ then $x = y$.
\end{lemma}
\begin{proof} Put $z = xy = uv$.
There exists $\mathbf{r} \in \mathbb{N}^{k}$ such that $\delta (x) = \delta (u) + \mathbf{r}$.
By the (UFP), we have that $x = u't$ where $\delta (u) = \delta (u')$
and $\delta (t) = \mathbf{r}$.
We therefore have that $u'ty = uv$.
We now apply the (UFP) again, to deduce that $u' = u$ and $v = ty$. 
If $\delta (x) = \delta (u)$ then we get $x = u$ from $x = ut$.
\end{proof}

%%%%%%%%%%%%%%%%%%%%%%%%%%%%%%%%%%%%%%%%%%%%%%%%%%%%%%%%%%%%%%%%%%%%%%%%%%%%%%%%%%%%%%%%%%%%%%%%%%%%%%%%%%%%%%%%%%%%%%%%%
We return briefly to $1$-monoids.
Let $S$ be a free monoid on the alphabet $A$.
The associated monoid homomorphism $\delta \colon S \rightarrow \mathbb{N}$ is always surjective
except in the case where $A$ is empty.
This can be generalized.

\begin{lemma}\label{lem:teeth} Suppose that $\delta \colon S \rightarrow \mathbb{N}^{k}$ is a $k$-monoid.
If $\delta$ is not surjective then there exists a monoid homomorphism $\delta' \colon S \rightarrow \mathbb{N}^{k-1}$
such that $S$ is a $(k-1)$-monoid.
\end{lemma}
\begin{proof} Suppose that $\delta$ is not surjective.
Then it is easy to see that there exists an element of $\mathbf{m} \in \mathbb{N}^{k}$ which is not in the image of $\delta$ 
and which has a component which is zero.
By permuting the components if necessary,
we can assume, without loss of generality, that it is the $k$th component of $\mathbf{m}$ which is $0$.
Observe, by the (UFP), that no element in the image of $\delta$ can have a non-zero entry in the $k$th position;
We may therefore define $\delta' \colon S \rightarrow \mathbb{N}^{k-1}$ to be $\delta$ followed by the map from $\mathbb{N}^{k} \rightarrow \mathbb{N}^{k-1}$
given by $(n_{1},\ldots, n_{k-1},n_{k}) \mapsto (n_{1},\ldots, n_{k-1})$.
Then $\delta'$ is a monoid homomorphism and satisfies the (UFP). 
\end{proof}

In the light of Lemma~\ref{lem:teeth},
we say that a $k$-monoid $S$ is {\em strict} if $\delta \colon S \rightarrow \mathbb{N}^{k}$
is surjective.
The following is now immediate by the above lemma and Corollary~\ref{cor:kfree-is-kmonoid}.

\begin{corollary}\label{cor:strict} Let $A_{1}, \ldots, A_{k}$ be $k$ alphabets each with at least one element.
Then $A_{1}^{\ast} \times \ldots \times A_{k}^{\ast}$ is a strict, singly aligned $k$-monoid.
\end{corollary}

%%%%%%%%%%%%%%%%%%%%%%%%%%%%%%%%%%%%%%%%%%%%%%%%%%%%%%%%%%%%%%%%%%%%%%%%%%%%%%%%%%%%%%%%%%%%%%%%%%%%%%%%%%%%%%%%%%%%%%%%%%%%%%%%%%%%%%%%
An {\em atom} in a monoid is a non-invertible element $a$ such that if $a = bc$ then at least one of $b$ or $c$ is invertible.

\begin{lemma}\label{lem:atoms} Let $S$ be a $k$-monoid.
\begin{enumerate}
\item The atoms in $S$ are the elements $a$ where $\delta (a) = \mathbf{e}_{i}$.
\item Every non-identity element is a product of atoms
\end{enumerate}
\end{lemma}
\begin{proof} (1) By the (UFP), it is clear that if $a$ is an atome then  $\delta (a) = \mathbf{e}_{i}$.
Suppose that $a$ is such that $\delta (a) =  \mathbf{e}_{i}$.
We prove that $a$ is an atom.
Suppose that $a = bc$.
Then $\delta (a) = \delta (b) + \delta (c)$.
But  $\delta (a) = \mathbf{e}_{i}$.
Thus either $\delta (b) = \mathbf{0}$ or $\delta (c) = \mathbf{0}$.
It follows that either $b$ is invertible or $c$ is invertible (in fact, the identity).

(2) Let $a$ be a non-identity element.
Then we may write $\delta (a) = m_{1}\mathbf{e}_{1} + \ldots + m_{k}\mathbf{e}_{k}$
where $m_{1},\ldots,m_{k}$ are natural numbers.
Using the fact that  $m_{1}\mathbf{e}_{1} = \mathbf{e}_{1} + \ldots + \mathbf{e}_{1}$ and so on and the (UFP),
we may now use (1) and deduce that every non-identity element is a product of atoms. 
\end{proof}

In free monoids $A^{\ast}$, the atoms are nothing other than the letters.
But, whereas free monoids have one alphabet, $k$-monoids have $k$,
which we now define.
Fix $k$ and let $1 \leq l \leq k$.
Define $\pi_{l} \colon \mathbb{N}^{k} \rightarrow \mathbb{N}$
by $(m_{1},\ldots, m_{l}, \ldots ,m_{k}) \mapsto m_{l}$.
This is a monoid homomorphism.
Let $\delta \colon S \rightarrow \mathbb{N}^{k}$ be a $k$-monoid.
Define $S_{l}$ to consist of all elements $a \in S$ such that $\delta (a)$ has $0$'s eveywhere except possibly at the $l$th component.
The set $S_{l}$ is non-empty since it must contain the identity.
Thus $S_{l}$ is clearly a submonoid of $S$.
Observe that $\delta_{l} = \pi_{l}\delta \colon S_{l} \rightarrow \mathbb{N}$
shows that $S_{l}$ is a $1$-monoid, and so free.
If $s \in S$ then $\delta (s) = m_{1}\mathbf{e}_{1} + \ldots + m_{k}\mathbf{e}_{k}$.
We may write $S = S_{1} \dots S_{l}$ uniquely.
It follows that $k$-monoids are constructed from $k$ free monoids (though not using direct products in general).
In the case $k = 2$, we can construct $2$-monoids from certain Zappa-Sz\'ep products \cite{Brin2005}
of free monoids.
For each $1 \leq l \leq k$, define $X_{l} = \delta^{-1}(\mathbf{e}_{l})$.
We call $(X_{1}, \ldots, X_{k})$ the {\em $k$ alphabets} associated with the $k$-monoid $S$.
Each set $X_{l}$ (which could be empty) is a set of free generators of the free monoid $S_{l}$
and consists of atoms by part (1) of Lemma~\ref{lem:atoms}.

The proof of the following lemma is immediate by the (UFP);
it shows that in a $k$-monoid there are certain relations between the atoms.

\begin{lemma}\label{lem:squares} Let $S$ be a $k$-monoid with the above notation. 
If $a$ and $b$ are atoms with $a \in X_{i}$ and $b \in X_{j}$, where $i \neq j$, 
then there exist unique atoms $a' \in X_{i}$ and $b' \in X_{j}$ such that $ab = b'a'$.
\end{lemma}

The relations guaranteed by the above lemma can be pictured geometrically as follows:

$
\centerline{
\xymatrix{
\ar[r]^{b}&   &   \\
\ar[u]^{a} & &  \\
}}$

\noindent
implies

$
\centerline{
\xymatrix{
\ar[r]^{b}&   &   \\
\ar[u]^{a} \ar@{-->}[r]_{b'} & \ar@{-->}[u]_{a'}&  \\
}}$

\noindent
since $ab = b'a'$,
and 

$
\centerline{
\xymatrix{
&   &   \\
\ar[r]_{b} & \ar[u]_{a}&  \\
}}$

\noindent
implies

$
\centerline{
\xymatrix{
\ar@{-->}[r]^{b'}&   &   \\
\ar@{-->}[u]^{a'} \ar[r]_{b} & \ar[u]_{a}&  \\
}}$

\noindent
since $ba = a'b'$.

It follows immediately from the above lemma, that $X_{i}X_{j} \subseteq X_{j}X_{i}$ for all $i \neq j$.
We now describe some special cases.

\begin{proposition}  Let $S$ be a $k$-monoid with $k$ alphabets $(X_{1}\ldots, X_{k})$.
\begin{enumerate}

\item  Suppose that  $a \in X_{i}$ and $b \in X_{j}$, where $i \neq j$, implies that $ab = ba$.
Then $S$ is isomorphic to a finite direct product of free monoids.

\item If $|X_{1}| = |X_{2}| = \ldots = |X_{k}| = 1$ then  $S \cong \mathbb{N}^{k}$.
\end{enumerate}
\end{proposition}
\begin{proof} (1)
It is enough to prove that $S \cong S_{1} \times \ldots \times S_{k}$.
For each $s \in S$, let $s = u_{1} \ldots u_{k}$ be the unique 
representation of $s$ as a product of elements of $S_{i}$ where $1 \leq i \leq k$.
Define $\theta (s) = (u_{1},\ldots, u_{k})$.
Clearly, this is a bijection.
Suppose that $t = v_{1} \ldots v_{k}$.
Then $st =  u_{1} \ldots u_{k}v_{1} \ldots v_{k}$.
By our assumption, $v_{1}$ commutes with all the elements $u_{2},\ldots, u_{k}$.
Thus we may write $st = (u_{1}v_{1})u_{2} \ldots u_{k}v_{2} \ldots v_{k}$.
Repeating, we obtain that $st = (u_{1}v_{1})(u_{2}v_{2}) \ldots (u_{k}v_{k})$.
This calculation shows that $\theta$ is a homomorphism.

(2) This is immediate by (1) above and the fact that the free monoid on one generator is isomorphic to $\mathbb{N}$.
\end{proof}

We want to regard elements of $k$ monoids as being $k$-dimensional boxes.
The following example motivates this idea.

\begin{example}\label{ex:scam}
{\em 
We shall now motivate our geometric representation of elements of $k$-monoids.
Let's take $k = 2$ and work with the $2$-monoid $S = \{a,b\}^{\ast} \times \{\alpha, \beta\}^{\ast}$,
a product of two free monoids on two generators a piece.
An element of this monoid is $x = (ab,\alpha\beta\beta)$.
This looks like two 1-dimensional strings.
We will show you that it is better to think of this as a $2 \times 3$ rectangle instead.
To save on notation, we shall write $a$ instead of $(a,1)$ and $\alpha$ instead of $(1,\alpha)$, and so on.
Our string is therefore $x = ab\alpha \beta \beta $.
But this is not the only way to write this string in terms of the generators.
To see why, we first set up a $2 \times 3$ grid:

$
\centerline{
\xymatrix{
 \ar[r]&  \ar[r] &  &  \\
\ar[u] \ar[r]&\ar[u] \ar[r]  &\ar[u] &  \\
\ar[u] \ar[r]  &\ar[u] \ar[r]  & \ar[u] & \\
\ar[u] \ar[r] & \ar[u] \ar[r] & \ar[u]   \\
}}$
We shall adopt the folloing convention:
horizontal lines will denote the elements $a$ and $b$
and vertical lines will denote the elements $\alpha$ and $\beta$.

$
\centerline{
\xymatrix{
 \ar[r]&  \ar[r] &  &  \\
\ar[u] \ar[r]&\ar[u] \ar[r]  &\ar[u]_{\beta} &  \\
\ar[u] \ar[r]  &\ar[u] \ar[r]  & \ar[u]_{\beta} & \\
\ar[u] \ar[r]_{a} & \ar[u] \ar[r]_{b} & \ar[u]_{\alpha}   \\
}}$

\noindent
We now traverse any path from the bottom left to the top right (double arrows):

$
\centerline{
\xymatrix{
 \ar[r]&  \ar[r] & \bullet  &   \\
\ar[u] \ar[r]& \ar[u] \ar@{=>}[r]  &\ar@{=>}[u]_{\beta} &  \\
\ar[u] \ar@{=>}[r]  &\ar@{=>}[u] \ar[r]  & \ar[u]_{\beta} & \\
\bullet \ar@{=>}[u] \ar[r]_{a} & \ar[u] \ar[r]_{b} & \ar[u]_{\alpha}   \\
}}$

\noindent
This corresponds to the product of generators $\alpha a \beta b \beta$
which is equal to the string $x$.
We therefore obtain all possible ways of representing the string $x$
as paths from the bottom left to the top right: and so 
from $ab\alpha \beta \beta$ to $\alpha \beta \beta a b$.
Each of these is a correct way of representing the string $x$
as a product of generators.
It follows that we should think of $x$ as the whole $2 \times 3$ rectangle.
}
\end{example}

%%%%%%%%%%%%%%%%%%%%%%%%%%%%%%%%%%%%%%%%%%%%%%%%%%%%%%%%%%%%%%%%%%%%%%%%%%%%%%%%%%%%%%%%%%%%%%%%%%%%%%%%%%%%
We shall now generalize the previous example.

We start with free monoids.
Let $A$ be a (non-empty) alphabet.
We shall now regard elements of $A$ as directed (to the right) line segments of unit length
labelled by an element of $A$.
A (non-empty) string $a$ over $A$ can now be regarded as the concatenation of
$|a|$ such directed line segments each of length $1$.
Thus the total length of the line segment that results is $|a|$.
We therefore regard the elements of $A^{\ast}$ as being labelled lines.

Suppose now that $k = 2$.
Let $S$ be a $2$-monoid.
There are two alphabets $X_{1}$ and $X_{2}$
and each element of $a \in S$ can be written uniquely
as $a = uv$ where $u \in X_{1}^{\ast}$ and $v \in X_{2}^{\ast}$.
It is therefore tempting to view $a$ as consisting of two strings over different alphabets.
But we shall argue (just as in our previous example) that it makes more sense to regard $a$ as a {\em rectangle} with area $m \times n$
where $\delta (a) = (m,n)$.
Observe, first, that any representation of $a$ as a product of atoms must contain
exactly $m$ atoms from $X_{1}$ and exactly $n$ atoms from $X_{2}$.
None of these ways of representing $a$ is privileged.
Represent elements of $X_{1}$ by horizontal line segments directed to the right
and represent elements of $X_{2}$ by vertical line segments directed upwards.
Then each way of representing an element $a$ can be regarded as a sequence of directed line segments
some horizontal and some vertical.
These can be labelled by elements of $X_{1} \cup X_{2}$.
We therefore obtain an $m \times n$ grid in which the horizontal line segments are labelled
by elements of $X_{1}$ and the vertical line segments are labelled by elements of $X_{2}$.

We can use this geometrical way of regarding the elements of $k$-monoids
to demonstrate how the multiplication works; 
of course, this can be done purely algebraically, but we want to show how the product is defined
{\em geometrically}.
We are given $x$ and $y$ inside a $2$-monoid, and our goal is to calculate $xy$ geometrically.
We place the top right-hand corner of $x$ against the bottom left-hand corner of $y$:
\begin{center}
\begin{tikzpicture}
\draw(0,0) rectangle (1.5,1);
\draw(1.5,1) rectangle (2.5,2);
\node at (0.75, 0.5) {$x$};
\node at (2, 1.5) {$y$};
\end{tikzpicture}
\end{center}
We do not yet have a rectangle, so we use the diagrams following Lemma~\ref{lem:squares} to fill in the gaps:
\begin{center}
\begin{tikzpicture}
\draw(0,0) rectangle (1.5,1);
\draw(1.5,1) rectangle (2.5,2);
\draw[dashed] (0,1) -- (0,2);
\draw[dashed] (0,2) -- (1.5,2);
\draw[dashed] (1.5,0) -- (2.5,0);
\draw[dashed] (2.5,0) -- (2.5,1);
\node at (0.75, 0.5) {$x$};
\node at (2, 1.5) {$y$};
\end{tikzpicture}
\end{center}
and now we get the result $xy$ represented as a rectangle
\begin{center}
\begin{tikzpicture}
\draw(0,0) rectangle (2.5,2);
\node at (1.25, 1) {$xy$};
\end{tikzpicture}
\end{center}
What we have said for $k = 2$ applies to any value of $k$.
We can therefore regard the elements of a $k$-monoid as being $k$-dimenional boxes
which we shall call {\em $k$-strings}.
It is in this way, that we can regard $k$-monoids as being {\em higher dimensional free monoids}.

%%%%%%%%%%%%%%%%%%%%%%%%%%%%%%%%%%%%%%%%%%%%%%%%%%%%%%%%%%%%%%%%%%%%%%%%%%%%%%%%%%%%%%%%%%%%%%%%%%%%%%%%%%%%%%%%%%%%%%%%%%%%%%%%%%%
\begin{example}{\em  In this example, we shall construct a specific monoid $S$ that
has all the hallmarks of a $3$-monoid but is, in fact, not.
Let $S$ be the monoid generated by the set $X_{1} \cup X_{2} \cup X_{3}$
where
$$X_{1} = \{a,a_{1},a_{2},a_{3}\}, \quad X_{2} = \{b,b_{1},b_{2},b_{3},b_{4}\},  \text{ and } X_{3} = \{c,c_{1},c_{2},c_{3}\}.$$
In addition to the associative law, the following relations are satisfied
$ab = b_{1}a_{1}$, $bc = c_{2}b_{2}$, $a_{1}c = c_{1}a_{2}$, $ac_{2} = c_{3}a_{3}$, $a_{3}b_{2} = b_{3}a_{2}$
and $b_{1}c_{1} = c_{3}b_{4}$;
in addition, all relations of the form 
$xy = yx$ where $x \in X_{i}$ and $y \in X_{j}$ and $i \neq j$ are satisfied
except where $xy = ab$ or $xy = bc$ or $xy = a_{1}c$ or $xy = ac_{2}$ or $xy = a_{3}b_{2}$ 
or $xy = b_{1}c_{1}$.
The key point is that we have specified $xy = y'x'$ in all cases where $x \in X_{i}$ and $y \in X_{j}$ where $i \neq j$.
We shall prove that $S$ is not a $3$-monoid.
We calculate $abc$ in two ways (remember, we are assuming associativity).
First, $abc = (ab)c = b_{1}a_{1}c = b_{1}c_{1}a_{2}$.
Second, $abc = a(bc) = ac_{2}b_{2} = c_{3}a_{3}b_{2}$.
By associativity, we must have that $b_{1}c_{1}a_{2}  = c_{3}a_{3}b_{2}$.
Now $(b_{1}c_{1})a_{2} = c_{3}b_{4}a_{2} = c_{3}a_{2}b_{4}$.
In $S$, we therefore have that $c_{3}a_{3}b_{2} = c_{3}a_{2}b_{4}$.
If $S$ were a $3$-monoid, it would be cancellative.
This would imply that $a_{3}b_{2} = a_{2}b_{4}$.
But this contradicts the fact that in a $3$-monoid
such a factorization would have to be unique.}
\end{example}

On the basis of the above example, there is one other property that will be crucial when we come to describe presentations of $k$-monoids.

\begin{lemma}\label{lem:associativity} Let $S$ be a $k$-monoid with $k$ alphabets $(X_{1},\ldots, X_{k})$.
Let $X_{i}$, $X_{j}$ and $X_{k}$ be distinct alphabets.
Let $f \in X_{i}$, $g \in X_{j}$ and $h \in X_{k}$.
Suppose that
$fg = g^{1}f^{1}$, $f^{1}h = h^{1}f^{2}$, $g^{1}h^{1} = h^{2}g^{2}$, $gh = h_{1}g_{1}$, $fh_{1} = h_{2}f_{1}$ and $f_{1}g_{1} = g_{2}f_{2}$.
Then we must have that $f^{2} = f_{2}$, $g^{2} = g_{2}$, and $h^{2} = h_{2}$.
\end{lemma}

%%%%%%%%%%%%%%%%%%%%%%%%%%%%%%%%%%%%%%%%%%%%%%%%%%%%%%%%%%%%%%%%%%%%%%%%%%%%%%%%%%%%%%%%%%%%%%%%%%%%%%%%%%%%%%%
We shall now work quite generally
although motivated by the Lemmas~\ref{lem:squares} and \ref{lem:associativity}.
Let $X$ be any non-empty set partitioned into $k$-blocks $X_{1}, \ldots, X_{k}$, each of which is assumed non-empty.
We shall refer to the elements of block $X_{i}$ as having the `colour $i$'.
For each ordered pair $(a,b) \in X_{i} \times X_{j}$, where $i \neq j$ (thus $a$ and $b$ have different colours),
we suppose that we are given a unique ordered pair $(b',a') \in X_{j} \times X_{i}$.
Put $R$ equal to the binary relation on $X^{\ast}$ which contains precisely the ordered pairs $(ab,b'a')$ and $(b'a',ab)$.
We call such a binary relation $R$ a {\em complete set of squares over $X$}.
We need an extra condition on $R$:
\begin{center}
$(ab,b^{1}a^{1}), (a^{1}c,c^{1}a^{2}), (b^{1}c^{1},c^{2}b^{2}) \in R$\\
and \\
$(bc,c_{1}b_{1}), (ac_{1},c_{2}a_{1}), (a_{1}b_{1},b_{2}a_{2}) \in R$\\
$\Rightarrow$ \\
$a^{2} = a_{2}$, $b^{2} = b_{2}$ and $c^{2} = c_{2}$.
\end{center}
We call this the {\em associativity condition}.

The following result was proved in the more general setting of category theory in \cite{HRSW}. 
It gives a monoid presentation of strict $k$-monoids.

\begin{theorem}\label{thm:construction} Let $X$ be any non-empty set partitioned into $k$-blocks $X_{1}, \ldots, X_{k}$. 
Let $R$ be any complete set of squares over $X$ that satisfies the associativity condition.
Let $\rho$ be the congruence generated by $R$.
Then $S = X^{\ast}/\rho$ is a strict $k$-monoid, and every strict $k$-monoid is isomorphic to one of this type.
Specifically, the quotient map from $X^{\ast}$ to $S$
is injective on $X$, and so we may identify $X$ with its image in $S$. 
There is a unique monoid homomorphism $\delta \colon S \to \mathbb{N}^k$ such that $\delta(a) = \mathbf{e}_{i}$ whenever $a \in X^i$.
\end{theorem}

%%%%%%%%%%%%%%%%%%%%%%%%%%%%%%%%%%%%%%%%%%%%%%%%%%%%%%%%%%%%%%%%%%%%%%%%%%%%%%%%%%%%%%%%%%%%%%%%%%%%%%%%%%%%%%%%%%%%%%%%%%%
\begin{example}\label{ex:av1}
{\em How do we construct concrete examples of $k$-monoids?
Theorem~\ref{thm:construction} shows that $k$-monoids always exist,
whereas by Corollary~\ref{cor:strict}, we can use finite direct products of $k$ free monoids (on at least one generator) 
to construct strict, singly aligned $k$-monoids.
Free monoids themselves arise from finite strings over an alphabet.
But this does not answer the question of how to construct {\em concrete} examples
of $k$-monoids in general.
In Section~3 of \cite{LV}, it was shown that any group acting co-compactly on a product of $k$ trees
can be used to construct a higher-rank graph with the number of vertices equal to the number of orbits.
This means that if the group acts simply transitively\footnote{The group $G$ acts simply transitively on the set $X$ if for any two elements $x$ and $y$ in $X$ there is a {\em unique} $g$
such that $gx = y$.} then we can construct a $k$-monoid.
Simply transitive groups actions on products of trees were constructed in \cite{stix-av} and \cite{RSV},
and Example~3.8 of \cite{LV} gives an explicit instance of a $k$-monoid constructed in this way. }
\end{example}

%%%%%%%%%%%%%%%%%%%%%%%%%%%%%%%%%%%%%%%%%%%%%%%%%%%%%%%%%%%%%%%%%%%%%%%%%%%%%%%%%%%%%%%%%%%%%%%%%%%%%%%%%%%%%%
\section{How to construct Thompson-Higman type groups}

There are many things one might do with $k$-monoids,
the main goal of this section is to explain how to construct groups from certain kinds of $k$-monoids.
This work goes back to a paper by Birget \cite{Birget}, 
was developed using free monoids in \cite{Lawson2007, Lawson2007b}
and then generalized to classes of $k$-monoids in \cite{LV2020, LSV}.
The paper \cite{Lawson2018} provides a retrospective on the older work.
In the papers \cite{Lawson2007, Lawson2007b}, it was shown that the Thompson group
$G_{n,1}$\footnote{Sometimes the notation $V_{n,1}$ is used for these groups.} arises from the free monoid on $n$-generators.
In unpublished work, John Fountain described how the Thompson-Higman groups might be generalized
but a suitable generalization of free monoids was lacking.
A class of $k$-monoids provides just what we are looking for.

The intersection of principal right ideals in a free monoid, if non-empty, is always a principal right ideal.
More generally,
we say that a $k$-monoid $S$ is {\em finitely aligned} if the intersection $aS \cap bS$ is either
empty or finitely generated as a right ideal. 
The notion of a monoid being finitely aligned was first defined by Gould \cite{Gould}, and then further studied in \cite{CG}.
Independently, it became very important in the theory of higher-rank graphs and their $C^{\ast}$-algebras \cite{RS2005, JS2014, JS2018}.
It will be useful to have some definitions centred around principal right ideals.
If $xS \cap yS \neq \varnothing$, we say that $x$ and $y$ are {\em comparable}
whereas if  $xS \cap yS = \varnothing$, we say that $x$ and $y$ are {\em incomparable}.
Incomparable subsets of $k$-monoids are analogues of prefix codes in free monoids,
so we call incomparable sets {\em generalized prefix codes}.
We say that $a \in S$ is {\em dependent on} a set $X$,
if $au = xv$ for some $x \in X$ and $u,v \in S$.
A generalized prefix code $X$ is said to be {\em maximal} if
every element of $S$ is dependent on some element of $X$.
In free monoids, maximal prefix codes are precisely those prefix codes maximal in the above sense.

We investigate first the intersection of principal right ideals in complete generality.
The following is a folklore result known to many working on higher-rank graphs.
It is therefore important but not original. 

\begin{lemma}\label{lem:nelson} Let $S$ be a $k$-monoid.
Assume that $a$ and $b$ are comparable and let $c \in aS \cap bS$.
Then there exists an element $d \in S$ such that 
$c = dt$, for some $t \in S$, 
where
$d \in aS \cap bS$ and 
$\delta (d) = \delta (a) \vee \delta (b)$.
(Recall that $\delta (a), \delta (b) \in \mathbb{N}^{k}$, and so the join refers to the join of two elements in $\mathbb{N}^{k}$.
This operation was defined at the beginning of Section~3).
Thus 
$$aS \cap bS = \bigcup_{\substack{d \in aS \cap bS,\\ \delta (d) = \delta (a) \vee \delta (b)}} dS.$$
\end{lemma}
\begin{proof} By assumption, $c = ax = by$ for some $x,y \in S$.
It follows that $\delta (c) \geq \delta (a), \delta (b)$.
Thus $\delta (c) \geq \delta (a) \vee \delta (b)$.
We may therefore write $\delta (c) = (\delta (a) \vee \delta (b)) + \mathbf{m}$ for some $\mathbf{m} \in \mathbb{N}^{k}$.
Thus by the (UFP), we may write $c = dt$ where $\delta (d) = d(a) \vee d(b)$ and $\delta (t) = \mathbf{m}$.
Now $c = dt = ax$.
But $\delta (d) = \delta (a) \vee \delta (b) = \delta (a) + \mathbf{n}$ for some $\mathbf{n} \in \mathbb{N}^{k}$.
By the (UFP), we may therefore write $d = a'u$ where $\delta (a') = \delta (a)$ and $\delta (u) = \mathbf{n}$.
Thus $a'ut = ax$.
It now follows by the (UFP), that $a = a'$.
We have therefore proved that $d \in aS$.
We may similarly prove that $d \in bS$.
\end{proof}

Our goal is to construct a group from a suitable $k$-monoid.
We shall do this by going via inverse monoids.

Let $S$ be a $k$-monoid which is finitely aligned.
Let $R_{1}$ and $R_{2}$ be right ideals of $S$.
A function $\theta \colon R_{1} \rightarrow R_{2}$ is called a {\em morphism}
if $\theta (as) = \theta (a)s$ for all $a \in R_{1}$ and $s \in S$.
A bijective morphism is called an {\em isomorphism}.
Define $\mathsf{R}(S)$ to be the set of all isomorphisms between
the finitely generated right ideals of $S$.

\begin{lemma}\label{lem:king} Let $S$ be a finitely aligned $k$-monoid.
Then $\mathsf{R}(S)$ is an inverse monoid.
\end{lemma}
\begin{proof} Let $\theta \colon R_{1} \rightarrow R_{2}$ be an isomorphism between finitely generated right ideals
and let $XS \subseteq R_{1}$ be a finitely generated right ideal.
We prove that $\theta (XS)$ is a finitely generated right ideal.
In fact, we claim that $\theta (XS) = \theta (X)S$.
Let $y \in \theta (XS)$.
Then $y = \theta (xs)$ for some $x \in X$ and $s \in S$.
Thus, using the fact that $\theta$ is a morphism,
we have that $y = \theta (x)s$.
It follows that $\theta (XS) \subseteq \theta (X)S$.
The reverse inclusion also follows from the fact that $\theta$ is a morphism.
By assumption, we have that the intersection of any two finitely generated right ideals is either empty
or again a finitely generated right ideal.
These two facts are enough to prove that $\mathsf{R}(S)$ is an inverse semigroup.
It is a monoid because $1S = S$ and so as a right ideal $S$ itself
is finitely generated.
\end{proof}

Observe that $\mathsf{R}(S)$ contains a zero,
and so we cannot simply define our group  to be 
$\mathsf{R}(S)/\sigma$, where $\sigma$ is the minimum group congruence \cite{Lawson1998},
because then the group will be trivial.
In order to contruct interesting groups, we need to focus on `large' elements of $\mathsf{R}(S)$.
What this means is delivered by the following definition.

Let $T$ be an inverse monoid with zero.
A non-zero idempotent $e$ of $T$ to be {\em essential} if for any non-zero idempotent
$f$ we have that $ef \neq 0$.
We say an element $a \in T$ is {\em essential} if both idempotents $a^{-1}a$ and $aa^{-1}$ are essential.
The {\em essential part} of $T$, denoted by $T^{e}$, is the set of all essential elements of $T$.
It is an inverse monoid; see \cite{Lawson2007b}.
The key point is that the zero is not an essential element and so does
not belong to the essential part of our inverse monoid.
We now locate the essential elements of  $\mathsf{R}(S)$.

\begin{lemma} Let $S$ be a finitely aligned $k$-monoid. 
Let $R = XS$ be a finitely generated right ideal of $S$.
Then the identity function on $R$ is an essential idempotent of  $\mathsf{R}(S)$
if and only if 
every element of $S$ is dependent on an element of $X$.
\end{lemma} 
\begin{proof} Suppose, first, that the identity function on $R$ is an essential idempotent.
Let $a \in S$ be any element.
Then $aS$ is a right ideal.
It follows that the identity function on $aS$ is an idempotent.
By assumption, this has a non-zero product with the identity function on $R$.
It follows that $aS \cap R \neq \varnothing$.
Thus there is some element $u \in S$ such that $au \in R$.
Whence, $au = xv$ for some $x \in X$ and $v \in S$.
The converse is proved from the observation that the element $a$ is dependent on an element of $X$
if and only if $aS \cap R \neq \varnothing$.
\end{proof}

We can define the group associated with the finitely aligned $k$-monoid $S$ as follows:

\begin{center}
\fbox{\begin{minipage}{8em}
{$\mathscr{G}(S) = \mathsf{R}(S)^{e}/\sigma.$}
\end{minipage}}
\end{center}

The only problem with this group is that we can, in general, say nothing about it.
So, we shall now define an apparently different group using generalized prefix codes.

\begin{lemma}\label{lem:grimm} Let $S$ be a finitely aligned $k$-monoid.
Then if $aS \cap bS \neq \varnothing$ the right ideal $aS \cap bS$ is generated by a finite generalized prefix code.
\end{lemma}
\begin{proof} By assumption, $aS \cap bS = XS$ where $X$ is a finite set.
By Lemma~\ref{lem:nelson},
for each $x \in X$ there exists an element $d_{x}$ such that $x = d_{x}t$ for some $t \in S$,
where $d_{x} \in aS \cap bS$ and $\delta (d_{x}) = \delta (a) \vee \delta (b)$.
Put $D = \{d_{x} \colon x \in X\}$.
This is a finite set since $X$ is a finite set.
We claim that $aS \cap bS = DS$.
By design, we have that $DS \subseteq aS \cap bS$.
If $s \in aS \cap bS$ then $s = xu$ for some $x \in X$ and $u \in S$.
It follows that $s = d_{x}tu$ and so $s \in DS$.
We have therefore proved that $aS \cap bS = DS$. 
But any two elements of $D$ have the same size.
Suppose that $d,d' \in D$ are comparable.
Then $dp = d'q$ for some $p,q \in S$.
By Lemma~\ref{lem:levi}, it follows that $d = d'$.
\end{proof}

Let $S$ be a finitely aligned $k$-monoid.
We shall now define a set $a \sqcup b$ when $a,b \in S$.
If $a,b \in S$ define $a \sqcup b = \varnothing$ if $a$ and $b$ are incomparable,
otherwise, define $a \sqcup b$ to be any finite generalized prefix code
such that $aS \cap bS = (a \sqcup b)S$; such exists by Lemma~\ref{lem:grimm}.   
The proof of the following is straightforward and can be found in \cite[lemma 3.24]{LSV}.

\begin{lemma}\label{lem:intersection-projective} Let $S$ be a finitely aligned $k$-monoid.
Then the intersection of any two right ideals generated by a finite generalized prefix code
is either empty or generated by a finite generalized prefix code.
\end{lemma}

We can now define our second group associated with a finitely aligned $k$-monoid.
It arose from unpublished work of John Fountain.
A finitely generated right ideal of $S$ is said to be {\em projective} if it is generated by a finite generalized prefix code.
Recall that in free monoids all right ideals are generated by prefix codes.
Let $S$ be a  finitely aligned $k$-monoid.
Define $\mathsf{P}(S)$ to be all the isomorphisms between the finitely generated right ideals generated
by generalized prefix codes.
By Lemma~\ref{lem:intersection-projective}, $\mathsf{P}(S)$ is also an inverse monoid.
This leads to our second group associated with $S$:
\begin{center}
\fbox{\begin{minipage}{8em}
{
$\mathscr{G}'(S) \cong \mathsf{P}(S)^{e}/\sigma.$
}
\end{minipage}}
\end{center}

We potentially have two groups associated with a finitely aligned $k$-monoid.
We shall now prove that for a natural class of $k$-monoids,
the two groups we have defined are, in fact, isomorphic. 

We need one further assumption on our $k$-monoid $S$.
We shall assume that $S$ is strict.
This assumption has an important consequence.
Let $\mathbf{m} \in \mathbb{N}^{k}$ be any element.
Define $C_{\mathbf{m}}$ to be all the elements $a$ of $S$ such that $\delta (a) = \mathbf{m}$.
Because $\delta$ is surjective, the set $C_{\mathbf{m}}$ is non-empty.
Suppose that $a,b \in C_{\mathbf{m}}$ are comparable.
Then $au = bv$ where $u,v \in S$.
But $a$ and $b$ have the same size and so $a = b$ by Lemma~\ref{lem:levi}.
Thus,  $C_{\mathbf{m}}$ is a generalized prefix code.
We now show that it is maximal.
Let $a$ be any element of $S$.
Then we can find an element $u$ such that $\delta (au) \geq \mathbf{m}$.
It follows by the (UFP) that $au = bv$ where $\delta (b) = \mathbf{m}$.
We have therefore proved the following.

\begin{lemma}\label{lem:existence} Let $S$ be a strict $k$-monoid.
Then for each $\mathbf{m} \in \mathbb{N}^{k}$ we have that $C_{\mathbf{m}}$ is a maximal generalized prefix code. 
\end{lemma}

The proof of the following is the key to proving that the two groups we have defined are isomorphic.

\begin{lemma}\label{lem:max} Let $S$ be a strict $k$-monoid.
Then every essential, finitely generated right ideal of $S$ contains a right ideal
generated by a finite maximal generalized prefix code.
\end{lemma}
\begin{proof} Let $XS$ be an essential right ideal where $X$ is finite.
Let $\mathbf{m}$ be the join of the sizes of the elements of $X$;
this makes sense because $\mathbb{N}^{k}$ is a lattice.
We prove that $C_{\mathbf{m}}S \subseteq XS$.
Let $a \in C_{\mathbf{m}}$.
Since $XS$ is essential, there exist elements $u,v \in S$ such that
$au = xv$ for some $x \in X$.
By assumption, $\delta (a) \geq \delta (x)$.
Thus by Lemma~\ref{lem:levi}, we have that $a = xt$ for some $t \in S$.
We have proved that $C_{\mathbf{m}}S \subseteq  XS$.
\end{proof}

It can now be proved that that each element of $\mathsf{R}(S)^{e}$ extends an element of $\mathsf{P}(S)^{e}$ \cite[Lemma 7.9]{LV2020}.
This implies that the groups $\mathscr{G}'(S)$ and $\mathscr{G}(S)$ are isomorphic.
We now summarize what we have found.

\begin{theorem}\label{them:main-theorem} 
Let $S$ be a finitely aligned strict $k$-monoid. 
Then the groups $\mathsf{R}(S)^{e}/\sigma$ and $\mathsf{P}(S)^{e}/\sigma$
are isomorphic.
\end{theorem}

We denote by $\mathscr{G}(S)$ the group guaranteed by the above theorem
and call it the {\em group associated with $S$.}
This group is intimately connected with the structure of
the finite, maximal generalized prefix codes on $S$.

%%%%%%%%%%%%%%%%%%%%%%%%%%%%%%%%%%%%%%%%%%%%%%%%%%%%%%%%%%%%%%%%%%%%%%%%%%%%%%%%%%%%%%%%%%%%%%%%%%%%%%%%%%%%%%%%%%%%%%%%%%%%%%%%%%%%%%%%%%%%%%%%%%%%%%%%%%%%%
We can now construct some groups using the above theorem.
Let $A_{1}, \ldots, A_{k}$ be $k$ non-empty alphabets.
Then $A_{1}^{\ast} \times \ldots \times A_{k}^{\ast}$
is a strict, singly aligned $k$-monoid by Corollary~\ref{cor:strict}. 
Accordingly, we may construct the group $\mathscr{G}(A_{1}^{\ast} \times \ldots \times A_{k}^{\ast})$.
When $k = 1$, we are back to the Thompson-Higman groups constructed in \cite{Lawson2007,Lawson2007b}.
When $A = A_{1}$ contains $2$ elements,
and we take the $n$-fold direct product $A^{\ast} \times \ldots \times A^{\ast}$
then the group $\mathscr{G}(A^{\ast} \times \ldots \times A^{\ast})$ is the group $nV$, the higher dimensional 
Thompson-Higman group of Matt Brin \cite{Brin}.\footnote{The group $G_{2,1}$ is often denoted by $V$ alone.}
See \cite{LV2020} for details.

We do not show this here, but the groups we have constructed also occur as the groups of units
of Boolean inverse monoids \cite{Wehrung}.
To prove this requires us to generalize the right-infinite strings over an alphabet $A$ to higher dimensions.
How this is done is described in \cite{LV2020, LSV}.

%%%%%%%%%%%%%%%%%%%%%%%%%%%%%%%%%%%%%%%%%%%%%%%%%%%%%%%%%%%%%%%%%%%%%%%%%%%%%%%%%%%%%%%%%%%%%%%%%%%%%%%%%%%%%%%%%
\section{Further generalizations}

This section is more exploratory in nature.
We have seen in part (3) of Lemma~\ref{lem:alg-properties} that the group of units of a $k$-monoid is trivial.
We may generalize $k$-monoids in the way described in \cite{LV2022} by allowing the group of units to be non-trivial.
Let $S$ be a monoid. 
We say that a monoid homomorphism $\lambda \colon S \rightarrow \mathbb{N}^{k}$ is a {\em size map}
if $\lambda^{-1}(\mathbf{0})$ is precisely the group of units of $S$.\\ 

\noindent
{\bf Definition. }A monoid $S$ is said to be a {\em generalized $k$-monoid} if there is a size map
$\delta \colon S \rightarrow \mathbb{N}^{k}$ satisfying the {\em weak factorization property (WFP)}:
if $\delta (x) = \mathbf{m} + \mathbf{n}$ then there exist elements $x_{1}$ and $x_{2}$ of $S$ 
such that $x = x_{1}x_{2}$ where $\delta (x_{1}) = \mathbf{m}$ and $\delta (x_{2}) = \mathbf{n}$
and, furthermore, if $x_{1}'$ and $x_{2}'$ are any elements such that $x = x_{1}'x_{2}'$ 
where $\delta (x_{1}') = \mathbf{m}$ and $\delta (x_{2}') = \mathbf{n}$ then
$x_{1}' = x_{1}g$ and $x_{2}' = g^{-1}x_{2}$ for some invertible element $g$.\\

Observe that $k$-monoids are simply the generalized $k$-monoids having a trivial group of units.
The following result tells us what generalized $1$-monoids look like.

\begin{proposition} Let $S$ be a monoid equipped with a size map $\lambda \colon S \rightarrow \mathbb{N}$.
Then the following are equivalent:
\begin{enumerate}
\item $S$ is a generalized $1$-monoid with respect to $\lambda$.
\item $S$ is equidivisible, and $a \in S$ is an atom if and only if $\lambda (a) = 1$.
\end{enumerate}
\end{proposition}
\begin{proof} (1) implies (2).
We prove first that $S$ is equidivisible
(this follows the proof of \cite[Proposition 3.4 (2)]{LV2022}).
Suppose that $z = xy = uv$.
We have that $\lambda (z) = \lambda (x) + \lambda (y) = \lambda (u) + \lambda (v)$.
Suppose, first, that $\lambda (x) \geq \lambda (u)$.
Then  $\lambda (x) = \lambda (u) + p$, where $p$ is some integer.
By the (WFP), we may write $x = x_{1}x_{2}$ where $\lambda (x_{1}) = \lambda (u)$
and $\lambda (x_{2}) = p$.
We therefore have that
$z = x_{1}(x_{2}y) = uv$.
But $\lambda (x_{1}) =  \lambda (u)$ and so $\lambda (x_{2}y) = \lambda (v)$.
It follows that there is an invertible element $g$ such that
$u = x_{1}g$ and $v = g^{-1}x_{2}y$.
We have that $x = x_{1}x_{2} = ug^{-1}x_{2}$.
Put $t = g^{-1}x_{2}$.
Then $x = ut$ and $v = ty$.
The assumption that $\lambda (x) < \lambda (u)$
leads to $u = xt$ and $y = tv$, for some $t \in S$.
Thus $S$ is equidivisible.

Suppose that $\lambda (a) = 1$.
If $a$ is not an atom, then $a = bc$ where neither $b$ nor $c$ is invertible.
It follows that $\lambda (b) \geq 1$ and $\lambda (c) \geq 1$.
This is a contradiction.
It follows that $a$ is an atom.
Suppose that $a$ is an atom.
We prove that $\lambda (a) = 1$.
Suppose not.
Then $\lambda (a) \geq 2$.
It follows from the (WFP) that $a = a_{1}a_{2}$ 
where $\lambda (a_{1}) \geq 1$ and $\lambda (a_{2}) \geq 1$.
We have therefore factorized $a = a_{1}a_{2}$
where neither $a_{1}$ nor $a_{2}$ is invertible.

(2) implies (1).
We prove that $S$ is a generalized $1$-monoid with respect to $\lambda$.
Let $a \in S$ be any element.
If $a$ is invertible then $\lambda (a) = 0$.
Suppose that $a = gh = g_{1}h_{1}$ where $g,h,g_{1},h_{1}$ are all invertible.
Put $k = g^{-1}g_{1}$, an invertible element.
Then $gk = g_{1}$ and $k^{-1}h = h_{1}$.

We may therefore suppose that $a$ is not invertible.
Suppose that $a$ is an atom.
Then, by assumption $\lambda (a) = 1$.
Suppose that $a = a_{1}a_{2}$ where $\lambda(a_{1}) = 1$ and $\lambda (a_{2}) = 0$.
Then $a_{1}$ is an atom and $a_{2}$ is invertible.
Thus $a = a1 = a_{1}a_{2}$.
Now we use equidivibility.
$a_{1} = at$ and $1 = ta_{2}$
or 
$a = a_{1}t$ and $a_{2} = t1$, for some $t \in S$.
If we take the first case, both $a_{1}$ and $a$ are atoms and so $t$ is invertible.
It follows that $a_{1} = at$ and $a_{2} = t^{-1}1$. 

We may therefore assume that $\lambda (a) \geq 2$.
Since $b$ is not an atom, we may factorize $b = b_{1}b_{2}$
where neither $b_{1}$ nor $b_{2}$ is invertible.
Observe that $\lambda (b_{1}), \lambda (b_{2}) < \lambda (b)$.
It follows that we may write $a$ as a product of $\lambda (a)$ atoms.
Observe that the number of atoms that we can factor $a$ into must equal $\lambda (a)$.
Thus $a = a_{1} \ldots a_{s}$, a product of atoms, where $s = \lambda (a)$.
Suppose that  $s = m + n$.
Then we may factorize $a = bc$ where $b = a_{1} \ldots a_{m}$ and $c = a_{m+1} \ldots a_{s}$.
Suppose, now, that $a = de$ where $\lambda (d) = m$ and $\lambda (e) = n$.
Then we may write $d = d_{1} \ldots d_{m}$, all atoms, and $e = e_{m+1} \ldots e_{s}$, also all atoms.
Thus $a = d_{1} \ldots d_{m}e_{m+1} \ldots e_{s}$.
By \cite[Lemma 2.11 (2)]{Lawson2015}, there are invertible elements $g_{1},\ldots , g_{s-1}$ such that
$a_{1} = d_{1}g_{1}, \ldots, a_{m} = g_{m-1}^{-1}d_{m}g_{m}$,
and
$a_{m+1} = g_{m}^{-1}e_{s}, \ldots,  a_{s} = g_{s-1}^{-1}e_{s}$.
Thus $b = dg_{m}$ and $c = g_{m}^{-1}e$ where $g_{m}$ is invertible.   
\end{proof}

We can say a lot more about generalized $1$-monoids
in the case where we know they have atoms and are left cancellative  

\begin{lemma} Let $S$ be a left cancellative generalized $1$-monoid which is equipped with a surjective size map $\lambda \colon S \rightarrow \mathbb{N}$.
Let $X$ be a transversal of the $\mathscr{R}$-classes of the atoms.
Denote by $X^{\ast}$ the submonoid of $S$ generated by the elements of $X$. 
Denote by $G$ be the group of units of $S$.
Then $X^{\ast}$ is a free monoid and every element of $S$ can be written uniquely
as a product of an element of $X^{\ast}$ followed by an element of $G$.
\end{lemma}
\begin{proof} Our assumption that $\lambda$ is surjective means that there are atoms. 
It is easy to check that if $a$ is an atom and $g$ is any invertible element then both $ag$ and $ga$  are atoms.
Using left cancellation, it is easy to prove that if $a$ and $b$ are arbitrary elements then $a \, \mathscr{R} \, b$ if and only if $a = bg$.
Let $X$ be a transversal of the $\mathscr{R}$-classes of the atoms.
Denote by $X^{\ast}$ the submonoid of $S$ generated by the set $X$.
We may restrict $\lambda$ to $X^{\ast}$.
We shall prove that the (UFP) holds.
Let $a \in X^{\ast}$ such that $\lambda(a) = m$.
Then $a = x_{1} \ldots x_{m}$ where $x_{1}, \ldots, x_{m} \in X$ are atoms.
If $m = s + t$ then we may factorize $a = a_{1}a_{2}$ where $a_{1} = x_{1} \ldots x_{s}$
and $a_{2} = x_{s+1} \ldots x_{m}$ where $\lambda (a_{2}) = t$.
To prove uniqueness, 
suppose that $x_{1} \ldots x_{q} = y_{1} \ldots y_{q}$ where $x_{i},y_{i} \in X$ for all
$1 \leq i \leq q$.
Then by the (WFP), we have that
$x_{1} = x_{2}g_{1}$ and
$x_{2} \ldots x_{q} = g_{1}^{-1}y_{2} \ldots y_{q}$
where $g_{1}$ is an invertible element.
Now $x_{1}, x_{2} \in X$ and $x_{1} \mathscr{R} x_{2}$.
But $X$ is a transversal and so $x_{1} = x_{2}$
and by left cancellation $g_{1} = 1$.
Now repeat the above argument and we get the result.
It follows that $X^{\ast}$ is a free monoid on $X$.
Put $G$ equal to the group of units of $S$.
We claim that $S = X^{\ast}G$.
If $a$ is not invertible, then $a = a_{1} \ldots a_{m}$, a product of atoms.
Now $a_{1} = x_{1}g_{1}$ for some $x_{1} \in X$ and invertible element $g_{1}$.
Thus $a = x_{1}g_{1}a_{2} \ldots a_{m}$.
Now repeat and we eventually show that $a = x_{1} \ldots x_{m}g$.
Thus $S = X^{\ast}G$.
We now prove that every element of $S$ can be written uniquely as a product 
of an element in $X^{\ast}$ and $G$.
Without loss of generality, suppose that
$x_{1} \ldots x_{m} = y_{1} \ldots y_{m}g$ where
$x_{1}, \ldots, x_{m},y_{1},\ldots, y_{m} \in X$ and $g \in G$.
Using the (WFP), we have that
$x_{1} = y_{1}h$ and 
$x_{2} \ldots x_{m} = h^{-1}y_{2} \ldots y_{m}g$.
As before, we deduce that $x_{1} = y_{1}$ and $h$ is trivial.
By left cancellation we get that 
$x_{2} \ldots x_{m} = y_{2} \ldots y_{m}g$.
Repeating the above argument we get that $x_{i} = y_{i}$ for all $1 \leq i \leq m$
and $g = 1$.\end{proof}

In the terminology of \cite{Lawson2008}, 
we have proved that left cancellative generalized $1$-monoids with atoms 
are Zappa-Sz\'ep products of groups and free monoids;
or, equivalently, there is a {\em self-similar group action} of $G$ on the free monoid $X^{\ast}$.
In fact, as shown in  \cite{Lawson2008, Lawson2015}
every such generalized $1$-monoid can be constructed in this way.
In \cite{LV2022}, self-similar group actions were generalized from groups acting on free monoids
to {\em groups acting on $k$-monoids}.

%%%%%%%%%%%%%%%%%%%%%%%%%%%%%%%%%%%%%%%%%%%%%%%%%%%%%%%%%%%%%%%%%%%%%%%%%%%%%%%%%%%%%%%%%%%%%%
\begin{example}\label{ex:av2}
{\em How do we construct concrete examples of generalized $k$-monoids?
Several infinite series of groups acting simply transitively on products of $k$ trees were constructed in \cite{RSV}.
In fact, each such group action can be used to construct an explicit example of a generalized $(k-1)$-monoid.
We sketch out here how an example of such a group acting simply transitively on a product of three trees
can be used to construct a generalized 2-monoid.
This is merely to whet the reader's appetite for using geometry to construct monoids.
The group $\Gamma$ constructed in \cite{RSV} acts on the product of three trees of valencies 4, 6, and 8, respectively, with one orbit.
 Its generators and relations are as follows:
$$
\Gamma = \left\langle
\begin{array}{c}
a_1,a_2 \\ 
b_1,b_2,b_3 \\
c_1,c_2,c_3,c_4
\end{array}
\ \left| 
\begin{array}{c}
a_1b_1a_4b_2,  \ a_1b_2a_4b_4, \  a_1b_3a_2b_1, \\ 

a_1b_4a_2b_3,  \ a_1b_5a_1b_6, \ a_2b_2a_2b_6 \\

a_1c_1a_2c_8, \ a_1c_2a_4c_4, \ a_1c_3a_2c_2, \ a_1c_4a_3c_3, \\
a_1c_5a_1c_6, \ a_1c_7a_4c_1, \ a_2c_1a_4c_6, \ a_2c_4a_2c_7 \\

b_1c_1b_5c_4, \
b_1c_2b_1c_5, \
b_1c_3b_6c_1, \\

b_1c_4b_3c_6, \
b_1c_6b_2c_3, \
b_1c_7b_1c_8, \\

b_2c_1b_3c_2, \
b_2c_2b_5c_5, \
b_2c_4b_5c_3, \\

b_2c_7b_6c_4, \
b_3c_1b_6c_6, \
b_3c_4b_6c_3
\end{array}
\right.\right\rangle.
$$
If we take a Euclidean square for each relation of $\Gamma$, write down the relation and its boundary, identify letters with the same labels respecting orientation and
fill in cubes as soon as there is a boundary of an empty  3-cube, we get a polyhedron $P$ such that $\Gamma$ is its fundamental group. 
We note that $P$ consists of 24 cubes.
The generators of $\Gamma$ and their inverses form three alphabets: 
$$
A=\{a_1^{\pm}, a_2^{\pm}\}
\quad
B=\{b_1^{\pm}, b_2^{\pm},  b_3^{\pm}\}
\quad
C=\{c_1^{\pm}, c_2^{\pm},  c_3^{\pm}, c_4^{\pm}\}.$$
As was already mentioned in Example~\ref{ex:av1},
a group acting simply transitively on a product of $k$ trees leads to a $k$-monoid.
The relations of $\Gamma$ containing the elements of the alphabets of $B$ and $C$ alone define a 2-monoid $M$. 
The elements of the alphabet $A$ act on $M$ according to the cubes of $P$. 
For example, $a_{1}(c_{1}) = c_{4}$ and $a_{1}^{-1}(b_{1}^{-1}) = b_{2}^{-1}$,
see the following picture:

\begin{figure}[h!]
 \centering
  \includegraphics[width=7.5cm, angle=-90]{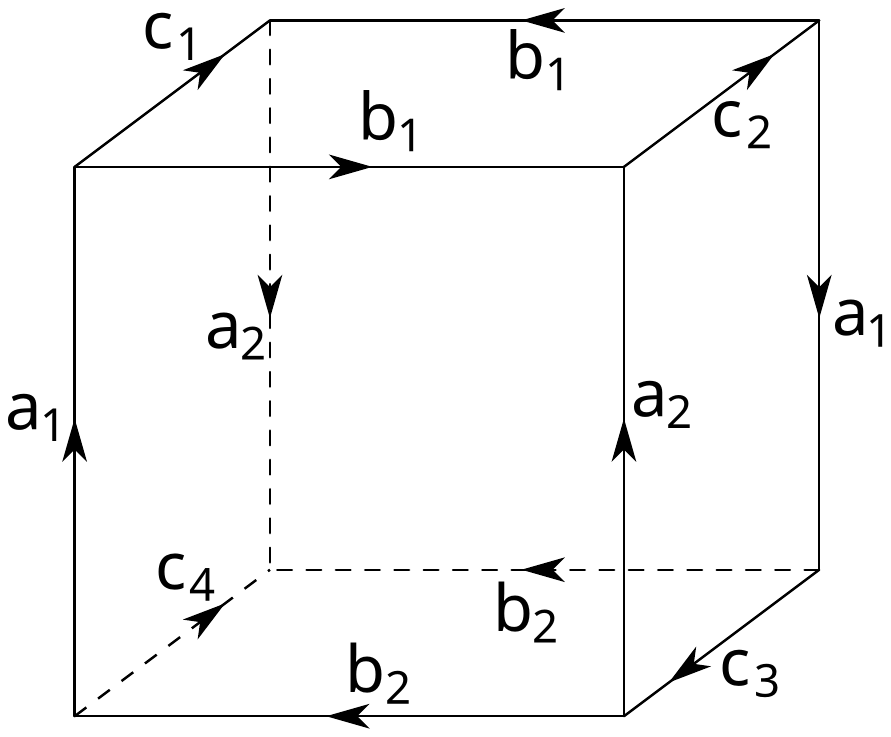}
\end{figure}

The action of $A$ on $M$ is well-defined,
 so
there is a group $G$ acting on $M$ in a self-similar way.
This leads to a generalized $2$-monoid.
However, we do not know what the group $G$ is in this case. 
The approach adopted in this example can be applied to any of the groups constructed in \cite{RSV}.
It would be interesting to know which groups of actions on $(k-1)$-monoids arise in this way.
}
\end{example}

%BIB

%%%%%%%%%%%%%%%%%%%%%%%%%%%%%%%%%%%%%%%%%%%%%%%%%%%%%%%%%%%%%%%%%%%%%%%%%%%%%%%%%%%%%%%%%%

\end{document}